\documentclass[english, 10pt]{amsart}
\usepackage{amsmath}
\usepackage{graphicx}
\usepackage{amssymb}
\usepackage{amsfonts, amsthm}
\usepackage[colorlinks=true,linkcolor=blue,citecolor=blue,urlcolor=blue]{hyperref}  

\newcommand{\ka}{\left(}
\newcommand{\kz}{\right)}

\newcommand{\R}{{\bar R}}

\newcommand{\al}{\alpha}
\newcommand{\be}{\beta}
\newcommand{\ga}{\gamma}

\newcommand{\auskommentieren}[1]{}

\newcommand{\beq}{\begin{equation}}
\newcommand{\eeq}{\end{equation}}
\newcommand{\bea}{\begin{equation}\begin{aligned}}
\newcommand{\eea}{\end{aligned}\end{equation}}

\newtheorem{theorem}{Theorem}[section]

\newtheorem{lem}[theorem]{Lemma}

\theoremstyle{definition}

\newtheorem{rem}[theorem]{Remark}

\newcommand{\q}{\quad}
\usepackage{array} 
\usepackage{amssymb}
\usepackage{latexsym}
\usepackage{epic}
\usepackage{curves}
\usepackage{fancyhdr}
\usepackage{xspace}

\numberwithin{equation}{section}

\DeclareMathOperator{\pr}{pr}

\DeclareMathOperator{\graph}{graph}

\begin{document}
\title{A note on inverse mean curvature flow in cosmological spacetimes}
\author{Heiko Kr\"oner}
\maketitle
\begin{abstract}
In \cite{CG} Gerhardt proves longtime existence for the inverse mean curvature flow in globally hyperbolic Lorentzian manifolds with compact Cauchy hypersurface, which satisfy three main structural assumptions: a strong volume decay condition, a mean curvature barrier condition and the timelike convergence condition. Furthermore, it is shown in \cite{CG} that the leaves of the inverse mean curvature flow provide a foliation of the future of the initial hypersurface. 

We show that this result persists, if we generalize the setting by leaving the mean curvature barrier assumption out. For initial hypersurfaces with sufficiently large mean curvature we can weaken the timelike convergence condition to a physically relevant energy condition.  
\end{abstract}

\tableofcontents

\section{Introduction and main result}
\footnote{{\bf Acknowledgment:} The author would like to thank Claus Gerhardt for stimulating discussions on the topic.}
\footnote{{\bf Address:} Mathematisches Institut, 
Universit\"at T\"ubingen,
Auf der Morgenstelle 10,
D-72076 T\"ubingen,
Germany. \\ {\bf email:} kroener@na.uni-tuebingen.de}
The inverse mean curvature flow has been considered by Huisken and Ilmanen \cite{HI} to prove the Penrose inequality or by Gerhardt \cite{G_euklid} in an Euclidean setting.

In \cite{CG} Gerhardt considers the inverse mean curvature flow in a globally hyperbolic Lorentzian manifold with compact Cauchy hypersurface, which satisfies the so-called timelike convergence condition, i.e. the Ricci tensor is nonnegative on the set of timelike unit vectors 
\bea \label{zeitartige_Kon} \R_{\al \be}
{\nu}^{\al}{\nu}^{\be} \ge 0 \q \forall \left<\nu, \nu\right> = -1, \eea 
a mean curvature barrier condition with respect to the future and a strong volume decay condition with respect to the future. For definitions we refer to \cite{CG}.

There it is shown, that the inverse mean curvature flow with connected, spacelike and closed initial surface with positive mean curvature exists for all times, that the flow hypersurfaces provide a foliation of the future of the initial surface and that the evolution parameter $t$ can be used as a new and physically meaningful time function.

We prove that all these results persist true, if we omit the assumption that a mean curvature barrier condition with respect to the future is fulfilled.

Furthermore we prove, that for initial hypersurfaces with sufficiently large mean curvature all above results stay true, if we relax the timelike convergence condition to the energy condition
\bea 
\label{ricci} \R_{\al \be} {\nu}^{\al}{\nu}^{\be} \ge -
\Lambda \q \forall \left<\nu, \nu\right> = -1 
\eea
for some $\Lambda > 0$.

In the following we assume that $N=N^{n+1}$ is a connected, time-oriented and smooth Lorentzian manifold. Furthermore, all functions will be assumed to be smooth; quantities like the Ricci tensor $\bar R_{\al \be}$, which refer to the ambient space (and not to hypersurfaces), are marked by a bar; greek indices range from 0 to n and latin indices from 1 to n. For a detailed introduction to our notations we refer to \cite{CG}. We formulate our result in two theorems.

\begin{theorem} \label{Hauptresultat}
Let $N$ be globally hyperbolic with connected and compact Cauchy hypersurface
$S_0$ und let $N$ fulfill a strong volume decay condition with respect to the future. Furthermore we assume that

\bea \bar R_{\alpha \beta}
{\nu}^{\alpha}{\nu}^{\beta} \ge -\Lambda \q \forall \left<\nu, \nu\right> =
-1 \eea 
with a positive constant $\Lambda$. Let $M_0$ be a connected, compact and spacelike hypersurface with mean curvature
\bea 
H > \sqrt{n\Lambda} 
\eea 
with respect to the past directed normal. Then the inverse mean curvature flow (IMCF) with initial surface $M_0$ exists for all times, i.e. given a manifold $M=M^n$ and an embedding
\bea 
x_0 : M \longrightarrow N 
\eea 
with $x_0(M) = M_0$, there is a 
\bea 
x: [0, \infty) \times M \longrightarrow N 
\eea 
with
\begin{itemize}
\item $x(t, \cdot)$ embedding of a spacelike hypersurface $M(t) = x(t, \cdot)\ka M\kz$ with positive mean curvature
\item  $x(0,\cdot) = x_0$
\item $\dot x(t, \xi) = -H^{-1}\nu$, where $\nu$ is the past directed normal of $M(t)$ in $x(t, \xi)$ and 
$H$ the mean curvature of $M(t)$ in $x(t, \xi)$ with respect to $\nu$.
\end{itemize}

The hypersurfaces $M(t)$, $t>0$, provide a foliation of the future $I^+\ka M_0\kz$ of $M_0$ and there holds
\bea 
\left|M(t)\right| = \left|M_0\right|e^{-t} \q \forall \ t \ge 0. 
\eea
Furthermore the evolution parameter $t$ can be used as new time function in the future
$I^+\ka M_0 \kz$ of $M_0$.
\end{theorem}

\begin{theorem} \label{thm2}
Let $t$ be the time function in $I^+(M_0)$ according to Theorem \ref{Hauptresultat}. If we define a second time function by setting
\bea 
\tau = 1- e^{-\frac{t}{n}}
\eea 
and use the obvious notation $M(t)=M(\tau)$
there holds
 \bea 
 |M(\tau)| = |M_0|(1-\tau)^{n} 
 \eea 
 for $0\le \tau < 1$. Furthermore there is a $0< \tau_0=\tau_0(n, \Lambda, M_0)<1$ and a constant $c=c(n, \Lambda, M_0)>0$ such that the length $L(\gamma)$ of a future directed timelike curve starting at 
$M(\tau)$ is bounded from above 
\bea 
\label{Laenge} L(\gamma) \le
c(1-\tau) \q\forall \tau_0\le \tau < 1. 
\eea 
\end{theorem}

\begin{rem}
Theorem \ref{thm2} shows that the quantity $1-\tau$ can be interpreted as the radius of the coordinate slices $\{\tau = const\}$ and as a measure for the remaining life span of the spacetime.
\end{rem}

To prove the above theorems we use a small but effective modification of the proof in \cite{CG}. We will only point out the differences between our proof and the one in \cite{CG}.

\section{Proof of Theorem \ref{Hauptresultat}}
We remark that under the assumptions of Theorem  \ref{Hauptresultat} the initial surface $M_0$ is a graph over $S_0$, cf. Section \ref{appendix}.

We need an evolution equation for $H^{-1}$.
\begin{lem} \label{Evolution_H_hoch_minus_eins} 
We have
 \bea
   \frac{d}{dt}\ka H^{-1}\kz- H^{-2} \Delta H^{-1} = - H^{-2}
\ka {\|A\|}^2+ {\bar R}_{\al\be}{\nu}^{\al}{\nu}^{\be}\kz H^{-1},
\eea 
where
\bea
 {\|A\|}^2 =h_{ij}h^{ij}. 
\eea
\end{lem}
\begin{proof}
See \cite[Lemma 2.2]{CG}.
\end{proof}
We show that the mean curvature of the leaves of the IMCF grows exponentially fast in time.
\begin{lem} \label{exponentiell} We assume that the IMCF exists on a maximal time interval $[0, T^*)$. Then there exists a constant $0<\epsilon = \epsilon(n, \Lambda, M_0)$ such that
\bea 
\label{mean_curvature_lower_bound} 
H \ge
e^{\epsilon t} \inf_{M_0}H \q\forall 0\le t<T^*. 
\eea
\end{lem}

\begin{proof}
We define
\bea 
f(t) = \frac{1}{n}\inf_{M(t)}H^2, \q 0 \le t <T^* 
\eea 
and choose $\Lambda < c < f(0)$. 
We show by contradiction that
\bea
\label{groesser_c} f(t)>c 
\eea
 for all $0\le t<T^*$. Assume this is not the case. Due to the continuity of $f$ there is a minimal $0<t_0<T^*$ with $f(t_0)=c$. For the function 
 \bea 
 \phi = H^{-1}e^{\epsilon t}, \q \epsilon = \frac{1}{n}\left(1-\frac{\Lambda}{c}\right),
 \eea 
 we have  for all
$0\le t\le t_0$ 
\begin{equation}\label{maximum} 
\begin{aligned}
\dot \phi - H^{-2}\Delta \phi \le&\
-H^{-2}{\|A\|}^2 \phi +
H^{-2}\Lambda\phi + \epsilon\phi \\
\le&\ -\frac{1}{n} \phi + \frac{\Lambda}{nc}\phi + \epsilon\phi \\
=&\ 0,
\end{aligned}
\end{equation}
cf. Lemma \ref{Evolution_H_hoch_minus_eins}. By applying the maximum principle we conclude
\bea 
\label{abschae} \phi (t, \cdot) \le \sup_{M_0}\phi = \sup_{M_0}H^{-1} 
\eea 
for all $ 0 \le t \le t_0$ and hence
\begin{equation}
\begin{aligned}
f(t_0) \ge&\ \frac{1}{n}e^{2\epsilon t_0} \inf_{M_0}H^2\\
>&\ c, 
\end{aligned}
\end{equation}
which is a contradiction.

With (\ref{groesser_c}) also (\ref{maximum}) and hence
(\ref{abschae}) is true for all $0\le t<T^*$ proving the lemma.
\end{proof}

The lower bound for the mean curvature of the leaves of the IMCF in Lemma \ref{exponentiell} can be improved.

\begin{lem} \label{exponentiell_besser}
We assume that the IMCF exists for all times. Then there exists a constant $c_0= c_0(n, \Lambda, M_0)>0$ such that 
\bea
H \ge c_0 e^{\frac{t}{n}}
\eea
during the evolution.
\end{lem} 
\begin{proof}
We define
\bea
w= H^{-1}e^{\frac{1}{n}t}
\eea
for all $t\ge 0$ and similar to inequality (\ref{maximum}) we get
\bea \label{1234}
\dot w - H^{-2}\Delta w \le H^{-2}\Lambda w.
\eea
The function
\bea
\varphi (t) = \sup w(t, \cdot)
\eea
is lipschitz continuous, cf. \cite[Lemma 3.2]{CG},  and for a.e. $t>0$ there holds
\bea
\dot \varphi(t) = \dot w(t, x_t),
\eea
where $x_t$ is an arbitrary point, in which the supremum is attained. Hence for a.e. $t>0$ we have in view of (\ref{1234}) and Lemma \ref{exponentiell}
\bea
\dot \varphi(t) \le & H^{-2}\Delta w + H^{-2}\Lambda w \\
 \le & H^{-2}\Lambda w \\
 \le & c e^{-2 \epsilon t} \varphi,
\eea
where the right side of the inequality is evaluated at $(t, x_t)$. Integration of the last inequality yields 
\beq
\varphi \le c,
\eeq
which proves the lemma.

\end{proof}

In \cite{CG} the assumption of barriers was used to show that the flow hypersurfaces of the IMCF run into the future singularity provided the flow exists for all times. In our proof we don't need barriers.

\begin{lem} \label{lemma_lower_order_estimates}
Let $N$ be as in Theorem \ref{Hauptresultat} and 
$S_0$, $(x^{\al})$ the corresponding Gaussian coordinate system, especially
\bea 
N = (a,b)\times S_0, \quad  a<0<b,
\eea 
as a topological product, $x^0$ a global defined time function, $(x^i)$ local coordinates for $S_0$ and $\{x^0=0\}$ corresponds to $S_0$.

If the IMCF exists for all times and if the leaves $M(t)$ of the IMCF  are graphs, i.e.
\bea 
\label{graph} M(t) = \left\{(x^0, x): x^0 = u(t,x),  \ x \in S_0\right\} 
\eea 
with a function $u\in C^{\infty}([0, \infty)\times S_0)$, then there holds
\bea
\label{beh} \lim_{t\rightarrow \infty} \inf_{S_0}u(t,\cdot) = b.
\eea
\end{lem}

\begin{proof}
For every $s \in (a,b)$, there is a $t_0 >0$  such that
\bea
\inf_{M(t)}H > \sup_{\{x^0=s\}}H 
\eea
for all $t \ge t_0$, cf. Lemma \ref{exponentiell_besser}. Applying \cite[Lemma 4.7.1]{CP} we conclude in view of Lemma \ref{exponentiell} that
\bea
u(t, \cdot) > s 
\eea
for all $t\ge t_0$. 
\end{proof}

The remaining part of Theorem \ref{Hauptresultat} is proved exactly as in \cite{CG}. 

\section{Proof of Theorem \ref{thm2}}

We need the following generalization by Anderson-Galloway, cf. \cite[Proposition
3.3]{Galloway}, of a singularity theorem of Hawking, cf.
\cite[Theorem 4, p. 272]{Hawking},  for globally hyperbolic
Lorentzian manifolds satisfying the timelike convergence condition to the case of globally hyperbolic Lorentzian manifolds satisfying (\ref{ricci}). 

\begin{theorem} \label{Hawking_allgemein}
Let $N=N^{n+1}$ be globally hyperbolic and let the Ricci tensor satisfy (\ref{ricci}). Let  $M$ be a compact,
spacelike and achronal hypersurface with mean curvature (with respect to past directed normal) 
\bea 
H\ge H_0 >\sqrt{n\Lambda}.
\eea 
Then the length $L(\ga)$ of an arbitrary future directed timelike curve $\ga$ starting on $M$
is bounded from above 
\bea \label{Laenge_kleiner} 
L(\ga) \le \frac{nH_0}{H_0^2-n\Lambda}. 
\eea
\end{theorem}
\begin{proof}
See \cite[Theorem 1.9.23]{CP}.
\end{proof}

We need the evolution equation for the induced metric $g_{ij}$ of the flow hypersurfaces $M(t)$.
\begin{lem}\label{evol_gij}
There holds
\bea
\dot g_{ij}=-2 H^{-1}h_{ij},
\eea
where $h_{ij}$ denotes the second fundamental form of the flow hypersurfaces.
\end{lem}
\begin{proof}
See \cite[Lemma 2.1]{CG}.
\end{proof}
For $g=\text{det}(g_{ij})$ we conclude from Lemma \ref{evol_gij} that 
\bea
\frac{d}{dt}\sqrt{g}=-\sqrt{g} 
\eea 
and therefore 
\bea
|M(t)|=|M_0|e^{-t} = |M_0|(1-\tau)^{n}. 
\eea 

To prove (\ref{Laenge}) we deduce from Lemma \ref{exponentiell_besser} that
\bea 
H|_{M(\tau)}=H|_{M(t)} \ge e^{\frac{t}{n}} c_0 = (1-\tau)^{-1} c_0 > \sqrt{n\Lambda} 
\eea
for all $\tau_0 \le \tau <1$ with $0<\tau_0<1$ suitable,
and apply Theorem \ref{Hawking_allgemein}.

\section{Appendix} \label{appendix}
\begin{lem} \label{Graph}
Let $N$, $(x^{\al})$ be as in Lemma \ref{lemma_lower_order_estimates} and $M\subset N$ a compact, connected and spacelike $C^m$-hypersurface. Then there exists $u \in C^m(S_0)$ so that 
\bea M
= \graph u |_{S_0}.
 \eea
\end{lem}
\begin{proof}

We combine the techniques in the proofs of \cite[Proposition 2.5]{Indiana} and \cite[Proposition 3.2.5]{CP}.

(i) The projection
\bea 
\pr: M &\longrightarrow& S_0, \q p= \ka
x^0(p), x(p) \kz \mapsto \ka 0, x(p)\kz 
\eea 
is continuous and injective, the latter will be shown in (ii).
Hence
 \bea 
 G \equiv \pr\ka M \kz \subset S_0 
 \eea 
 is closed and $M = \graph u|_G$ with a suitable $u \in C^0\ka G \kz$. We will show that
$G \subset S_0$ is also open ($\Rightarrow G = S_0$) and that $u
\in C^m\ka S_0 \kz$. For this let $z_0 \in G$ be arbitrary, $p = \ka
u(z_0), z_0 \kz \in M$ and $U$ a neighbourhood of $p$ in $N$, so that there is a $\varphi \in C^m\ka U \kz$ with 
\bea 
U \cap M = \left\{\ka x^0, x\kz \in U : \varphi \ka x^0, x \kz = 0\right\}. 
\eea 
Since $M$ is spacelike, $D\varphi$ is timelike and hence
\bea
\frac{\partial \varphi}{\partial x^0}= \left<D\varphi,
\frac{\partial}{\partial x^0}\right> \ne 0. 
\eea 
Due to the implicit function theorem there are neighbourhoods $V$ of $z_0$ in
$S_0$ and $\tilde U$ of $p$ in $U$ with 
\bea 
\tilde U \cap M = \graph \psi |_V, \q \psi \in C^m\ka V\kz. 
\eea 
This proves the lemma in view of $\psi = u|_V$.

(ii) We show that $\pr$ is injective. Let $\bar x \in M$ be so that
\bea 
x^0(\bar x) = \inf_{M} x^0, 
\eea 
where we assume w.l.o.g. that this infimum is strictly larger than 0.

For a point $p\in M$ we define 
\bea
 z(p) = \left\{ (t,\pr(p)):
0\le t < x^0(p)\right\}. 
\eea 

The set 
\bea 
\Lambda = \left\{ p
\in M : z(p) \cap M =\emptyset \right\} 
\eea 
is nonempty, since $\bar x \in \Lambda$ and closed and open as will be proven. Hence we deduce $\Lambda = M$ and the injectivity of $\pr$.

(a) For proving that $\Lambda$ is closed let $p_k \in
\Lambda$ be a sequence with $p_k \rightarrow p_0 \in M$. 
If $p_0 \notin \Lambda$, then there exists $\tilde p_0 \in z(p_0)\cap M$ and by implicit function theorem neighborhoods $U$ of
$\pr(p_0)$ in $S_0$ and $V$ of $\tilde p_0$ in $M$ as well as $u\in
C^m(U)$ with 
\bea 
V = \graph \ u|_U \q \wedge \q u < x^0(p_0). 
\eea 
This implies $p_k \notin \Lambda$ for a.e. $k$, a contradiction.

(b) Suppose that $\Lambda$ is not open, then there exists $p_0 \in
\Lambda$, a sequence $p_k \in M\backslash\Lambda$ with $p_k
\rightarrow p_0$ and a sequence $q_k \in z(p_k)\cap M$.

 Due to inverse function theorem there is a neighbourhood $U$ of $p_0$ in $N$, so that $U\cap M$ is a $C^m$-graph
over a suitable subset of $S_0$, which implies 
\bea
{\overline \lim_{k\rightarrow \infty}} x^0(q_k) < x^0(p_0). 
\eea
Hence a subsequence of the $q_k$ converges to a point in $z(p_0)\cap M$, a contradiction.
\end{proof}

\end{document}